\documentclass{amsart}
\usepackage[utf8]{inputenc}
\usepackage{amssymb,latexsym}
\usepackage{amsmath}
\usepackage{graphicx}
\usepackage{textcomp}
\usepackage[dvipsnames]{xcolor}

\usepackage{amsthm,amssymb,enumerate,graphicx, tikz}
\usepackage{amscd}
\usepackage{setspace}
\usepackage{comment}
\usepackage{hyperref}
\usepackage{cleveref}

\def\@logofont{\footnotesize}
\def\@setaddresses{\par
  \nobreak \begingroup
  \footnotesize
  \def\author##1{\nobreak\addvspace\bigskipamount}%
  \def\\{\par\nobreak}%
  \interlinepenalty\@M
  \def\address##1##2{\begingroup
    \par\addvspace\bigskipamount\indent
    \@ifnotempty{##1}{(\ignorespaces##1\unskip) }%
    {\scshape\ignorespaces##2}\par\endgroup}%
  \def\curraddr##1##2{\begingroup
    \@ifnotempty{##2}{\nobreak\indent\curraddrname
      \@ifnotempty{##1}{, \ignorespaces##1\unskip}\/:\space
      ##2\par}\endgroup}%
  \def\email##1##2{\begingroup
    \@ifnotempty{##2}{\nobreak\indent\emailaddrname
      \@ifnotempty{##1}{, \ignorespaces##1\unskip}\/:\space
      \ttfamily##2\par}\endgroup}%
  \def\urladdr##1##2{\begingroup
    \def~{\char`\~}%
    \@ifnotempty{##2}{\nobreak\indent\urladdrname
      \@ifnotempty{##1}{, \ignorespaces##1\unskip}\/:\space
      \ttfamily##2\par}\endgroup}%
  \addresses
  \endgroup
}
\renewcommand*\subjclass[2][2010]{%
  \def\@subjclass{#2}%
  \@ifundefined{subjclassname@#1}{%
    \ClassWarning{\@classname}{Unknown edition (#1) of Mathematics
      Subject Classification; using '2000'.}%
  }{%
    \@xp\let\@xp\subjclassname\csname subjclassname@#1\endcsname
  }%
}

\usepackage{graphicx,amsmath,amssymb}
\usepackage{algorithm}
\usepackage{mathdots, comment}
\usepackage[noend]{algpseudocode}

\newtheorem{theorem}{Theorem}[section]

\newtheorem*{theorem*}{Theorem}

\newtheorem{lemma}[theorem]{Lemma}

\newtheorem{corollary}[theorem]{Corollary}

\theoremstyle{definition}

\theoremstyle{remark}
\newtheorem{remark}[theorem]{Remark}
\newtheorem{example}[theorem]{Example}

\begin{document}
\title[Characterization of matchable sets and subspaces]{Characterization of matchable sets and subspaces via Dyson transforms}

\author[M. Aliabadi, J. Losonczy]{Mohsen Aliabadi$^{1}$ \and Jozsef Losonczy$^{2,*}$}
\thanks{$^1$Department of Mathematics, University of California, San Diego, 
9500 Gilman Dr, La Jolla, CA 92093, USA.  \url{maliabadisr@ucsd.edu}.\\
$^2$Department of Mathematics, Long Island University,
720 Northern Blvd, Brookville, New York 11548, USA. \url{Jozsef.Losonczy@liu.edu}.}
\thanks{$^*$Corresponding Author.}


\begin{abstract}
A \emph{matching} from a finite subset \( A \) of an abelian group \( G \) to another subset \( B \) is a bijection \( f : A \to B \) such that \( af(a) \notin A \) for all \( a \in A \). The study of matchings began in the 1990s and was motivated by a conjecture of E.~K.~Wakeford on canonical forms for homogeneous polynomials. The theory was later extended to the linear setting of vector subspaces over field extensions, and then to matroids. In this paper, we investigate the existence and structure of matchings in both abelian groups and field extensions. Using Dyson's \( e \)-transform, a tool from additive combinatorics, along with a linear analogue which is introduced in this paper, we establish characterization theorems for matchable sets and subspaces.  Several applications are given to demonstrate the effectiveness of these theorems as standalone tools. Throughout, we highlight the parallels between the group-theoretic and linear perspectives.
\end{abstract}

\maketitle

\section{Introduction}

\textbf{History of matchings.} A geometric framework for a class of bipartite graphs was introduced in~\cite{Losonczy 1}, where a special type of perfect matching, termed an acyclic matching, was defined and shown to exist for certain graphs using geometric techniques. The existence of such matchings is closely tied to the non-vanishing of determinants of specific weighted biadjacency matrices. This setup was applied to a conjecture of E.~K.~Wakeford~\cite{Wakeford} which dates to 1916 and involves determining the sets of monomials that can be eliminated from a generic homogeneous polynomial through linear changes of variables.

In a notable special case, Wakeford's conjecture was reduced to the problem of showing that acyclic matchings exist for certain pairs of subsets in \( \mathbb{Z}^n \). The full conjecture remains open, but the acyclic matching property introduced in~\cite{Losonczy 1} was later shown to hold in the most general sense for \( \mathbb{Z}^n \) by Alon et al.~\cite{Alon}, and was completely characterized for all abelian groups in \cite{Taylor}. Further developments appeared in~\cite{Losonczy 2}, where a broader class of matchings was investigated in abelian groups.  The theory was then extended to arbitrary groups by Eliahou and Lecouvey~\cite{Eliahou 1}. Related counting aspects were explored by Hamidoune~\cite{Hamidoune}.

A linear formulation of matchings was introduced in~\cite{Eliahou 2}, providing an analogue in the setting of field extensions, while a matroidal version was proposed in~\cite{Aliabadi 3}. A recent application of matchings in abelian groups within combinatorial number theory can be found in~\cite{Lev}.\\

\textbf{Organization of paper.} In the following two sections, we first revisit a pair of foundational results on matchings which will be used in the abelian group setting (Section~\ref{abeliab sec}), and then transition to the linear-algebraic framework, outlining the necessary background information on matchable subspaces over field extensions (Section~\ref{linear sec}). Building upon these backgrounds, Sections~\ref{MIBG} and~\ref{MSFE} present our main contributions, offering characterizations of matchable sets in abelian groups and matchable subspaces in field extensions, respectively, as stated in Theorems~\ref{CharForAbelianGroups} and~\ref{main linear}. We will also give several applications to show that these theorems are effective tools on a standalone basis, eliminating a longstanding reliance on a variety of inequalities from additive number theory and related areas.

\subsection{Preliminaries on matchings (abelian group setting)}\label{abeliab sec}

Let \( A \) and \( B \) be nonempty finite sets of the same cardinality, and let \( \mathcal{G} \) be a subset of \( A \times B \). A bijective mapping \( f : A \longrightarrow B \) is called a {\em matching} of \( \mathcal{G} \) if \( (a, f(a)) \in \mathcal{G} \) for all \( a \in A \).  Note that \( A \) and \( B \) are not required to be disjoint.

For \( S \subseteq A \) and \( T \subseteq B \), we define 
\[ \mathcal{G}_1(S) = \{ b \in B :  (a,b) \in \mathcal{G}\mbox{ for some }a \in S \}, \]
\[ \mathcal{G}_2(T) = \{ a \in A :  (a,b) \in \mathcal{G}\mbox{ for some }b \in T \}, \]
and for \( a \in A \) and \( b \in B \), let \(d_1(a) = |\mathcal{G}_1(\{ a \})| \) and \(d_2(b) = |\mathcal{G}_2(\{ b \})| \).  In the first part of the paper, where we consider matchings in abelian groups, we will make use of two well-known results from matching theory. For convenience, we state them here using the above notation. Proofs can be found in \cite{Lovasz}. The first is Philip Hall's marriage theorem, which gives a necessary and sufficient condition for the existence of a matching.

\begin{theorem}[P.\ Hall]  \label{PhiliplHall}
Let \( A \) and \( B \) be nonempty finite sets such that \( |A| = |B| \), and let \( \mathcal{G} \) be a subset of \( A \times B \). Then there exists a matching of \( \mathcal{G} \) if and only if for every nonempty subset \( S \) of \( A \), we have \( |S| \leq | \mathcal{G}_1(S)|. \)
\end{theorem}

The other result, attributed to Marshall Hall, is useful for establishing a lower bound on the number of matchings.

\begin{theorem}[M.\ Hall]  \label{MarshallHall}
Let \( A \) and \( B \) be nonempty finite sets such that \( |A| = |B| \), let \( \mathcal{G} \) be a subset of \( A \times B \), and let \( n \) be a positive integer. Assume that there exists at least one matching of \( \mathcal{G} \), and that for each \( b \in B \), we have \( d_2(b) \geq n \). Then there are at least \( n! \) matchings of \( \mathcal{G} \).
\end{theorem}

We are interested in a certain group-theoretic context for the sets \( A \), \( B \), and \( \mathcal{G} \). Let \( G \) be an abelian group (with operation written multiplicatively) and let \( A \) and \( B \) be nonempty finite subsets of \( G \) such that \( |A| = |B| \). Define \( \mathcal{G} \) by 
\[  \mathcal{G} = \{ (a,b) \in A \times B : ab \notin A \}. \]
In this paper, we will always assume that \( \mathcal{G} \) is as above. A matching of \( \mathcal{G} \) is then a bijection \( f : A \longrightarrow B \) satisfying \( af(a) \notin A \) for all \( a \in A \). We will usually not mention \( \mathcal{G} \) explicitly and instead refer to such an \( f \) as a ``matching from \( A \) to \( B \)."

\begin{example}
Let \( G \) be a cyclic group of order \( 6 \), and let \( x \) be a generator. Take \( A = \{ 1, x^2, x^4, x^5 \} \) and \( B = \{ x, x^2, x^3, x^4 \} \).  Then there is no matching from \( A \) to \( B \), since, for the subset \( S = \{ 1, x^2, x^4 \} \) of \( A \), we have \( \mathcal{G}_1(S) = \{ x, x^3 \} \), so that the condition in Theorem~\ref{PhiliplHall} is violated.  If instead we take \( B \) to be the set \( \{ x, x^2, x^3, x^5 \} \), then there are exactly two matchings from \( A \) to \( B \); one is given by \( 1 \mapsto x \), \( x^2 \mapsto x^5 \), \( x^4 \mapsto x^3 \), \( x^5 \mapsto x^2 \), and the other is the mapping \( 1 \mapsto x^3 \), \( x^2 \mapsto x \), \( x^4 \mapsto x^5 \), \( x^5 \mapsto x^2 \).
\end{example}

The number of matchings (possibly \( 0 \), as we just saw above) turns out to be related to the arithmetic structures of \( A \) and \( B \), as well as the algebraic structure of \( G \).  

It will be convenient to have a version of Theorem~\ref{PhiliplHall} which is tailored to our group-theoretic setup. Given a subset \( S \) of \( A \), let 
\[ U = \{ b \in B : Sb \subseteq A \}. \] 
Then \( B \setminus U = \mathcal{G}_1(S) \).  Therefore, we have the following:

\begin{corollary}  \label{RevisedPhilipHall}
Let \( G \) be an abelian group, and let \( A \) and \( B \) be nonempty finite subsets of \( G \) such that \( |A| = |B| \).  Then there exists a matching from \( A \) to \( B \) if and only if for every nonempty subset \( S \) of \( A \), we have \( |S| \leq |B \setminus U| \), where \( U = \{ b \in B : Sb \subseteq A \} \).
\end{corollary}

We conclude with the simple observation that a necessary condition for the existence of a matching from \( A \) to \( B \) is \( 1 \notin B \).  For \( A \) and \( B \) contained in certain abelian groups \( G \), this condition is also sufficient (see Corollary~\ref{MatchingPropertyGroup} for a precise statement).

\bigskip

\subsection{Preliminaries on matchings (linear setting)}\label{linear sec}
For any positive integer \( n \), we use \( [n] \) to denote the set \( \{1, \ldots, n\} \). Given a subset \( S \) of a vector space \( V \), we write \( \langle S \rangle \) for the subspace of \( V \) spanned by \( S \).  If \( S = \{ x_1, \ldots , x_n \} \), we may also denote this subspace by \( \langle x_1,  \ldots , x_n \rangle \).

The notion of matching two subspaces in a field extension, as described below, was introduced by Eliahou and Lecouvey in~\cite{Eliahou 2}.

Let \( K \subseteq L \) be a field extension, and let \( A \) and \( B \) be two \( n \)-dimensional \( K \)-subspaces of \( L \), with \( n > 0 \). An ordered basis \( \mathcal{A} = \{a_1, \ldots, a_n\} \) of \( A \) is said to be \emph{matched} to an ordered basis \( \mathcal{B} = \{b_1, \ldots, b_n\} \) of \( B \) if
\[
a_i^{-1} A \cap B \subseteq \langle b_1, \ldots, b_{i-1}, b_{i+1}, \ldots, b_n \rangle \quad \text{for each } i \in [n].
\]
We say that \( A \) is \emph{matched} to \( B \) (or is \emph{matchable} to \( B \)) if every ordered basis of \( A \) can be matched to some ordered basis of \( B \).

Note that if the above condition holds, then \( a_i b_i \notin \mathcal{A} \) for all \( i \), so the map \( a_i \mapsto b_i \) defines a matching, in the group-theoretic sense, from \( \mathcal{A} \) to \( \mathcal{B} \) in the multiplicative group \( L^\times \).

\begin{remark}\label{Necessary for Matching}
   A necessary condition for \( A \) to be matched to \( B \) is that \( 1 \notin B \). This is discussed in detail in \cite{Eliahou 2}; however, to keep the presentation here as self-contained as possible, we repeat their argument.
   
   Assume that \( A \) is matched to \( B \), and suppose, for the sake of contradiction, that \( 1 \in B \). Let \( \mathcal{A} = \{a_1, \dots, a_n\} \) be a basis of \( A \). Then \( \mathcal{A} \) is matched with a basis \( \mathcal{B} = \{b_1, \dots, b_n\} \). By definition, we have
\[
a_i^{-1}A \cap B \subseteq \langle b_1, \dots, b_{i-1}, b_{i+1}, \dots, b_n \rangle,
\]
for each \( i \in [n] \). This implies
\[
1 \in \bigcap_{i \in [n]} \left(a_i^{-1}A \cap B\right) \subseteq \bigcap_{i \in [n]} \langle b_1, \dots, b_{i-1}, b_{i+1}, \dots, b_n \rangle = \{0\},
\]
which is a contradiction.
\end{remark}

A field extension \( K \subseteq L \) is said to have the \emph{linear matching property} if for every pair of finite-dimensional \( K \)-subspaces \( A \) and \( B \) of \( L \) with \( \dim A = \dim B > 0 \) and \( 1 \notin B \), the subspace \( A \) is matched to \( B \).

\medskip

In the second part of the paper, where we examine matchings in the linear setting, we make use of an analogue of P.\ Hall's marriage theorem (Theorem~\ref{PhiliplHall}), expressed in the language of systems of distinct representatives.

Let \( V \) be a finite-dimensional vector space over a field \( K \), and let \( \mathcal{W} = \{ W_i \}_{i \in [n]} \) be a family of subspaces of \( V \). It is not assumed that the \( W_i \) are distinct.  A \emph{free transversal} for \( \mathcal{W} \) is a linearly independent set of vectors \( \{x_1, \ldots, x_n\} \subseteq V \) such that \( x_i \in W_i \) for each \( i \in [n] \).

A fundamental result of Rado~\cite{Rado} provides a necessary and sufficient condition for the existence of a free transversal, closely resembling the condition in P.\ Hall’s classical marriage theorem.

\begin{theorem}[Rado]\label{Linear Hall}
Let \( V \) be a finite-dimensional vector space over a field \( K \), and let \( \mathcal{W} =  \{ W_i \}_{i \in [n]} \) be a family of subspaces of \( V \). Then \( \mathcal{W} \) admits a free transversal if and only if
\[
\dim\left( \sum_{i \in J} W_i \right) \geq |J| \quad \text{for all } J \subseteq [n].
\]
\end{theorem}

Given a field extension \( K \subseteq L \) and \( K \)-subspaces \( A \) and \( B \) of \( L \), we use \( AB \) to denote the \emph{Minkowski product} of \( A \) and \( B \):
\[
AB = \{ ab : a \in A,\ b \in B \}.
\]

\medskip

By combining Rado's theorem with linear analogues of two theorems from additive number theory, Eliahou and Lecouvey~\cite{Eliahou 2} established the following fundamental results:
\begin{itemize}
    \item A subspace \( A \) is matched to itself if and only if \( 1 \notin A \).
    \item A field extension \( K \subseteq L \) has the linear matching property if and only if \( L \) contains no nontrivial proper finite-dimensional extension over \( K \).
\end{itemize}

The main objective of our work in the linear setting is to develop an efficient and unified framework for characterizing pairs of matchable subspaces, one that not only recovers known results but also provides a definitive perspective on the underlying structure of the pairs.

\section{Matchings in abelian groups} \label{MIBG}

Let \( G \) be an abelian group and let \( S \) be a subset of \( G \).  In this section, the notation \( \langle S \rangle \) is used for the subgroup of \( G \) generated by \( S \). If \( S = \{ x \} \), then we also write \( \langle x \rangle \) for this subgroup. We use \( o(x) \) for the order of any \( x \in G \), with the understanding that \( o(x) = \infty \) if \( x \) does not have finite order.  

The following is the first main result of this paper.  Its proof will rely on Dyson's \( e \)-transform, which is discussed in Chapter 2 of \cite{Nathanson}.
Note, however, that our particular use of the \( e \)-transform will not require any background knowledge concerning its properties.

\begin{theorem}  \label{CharForAbelianGroups}
Let \( A \) and \( B \) be nonempty finite subsets of an abelian group \( G \) such that \( |A| = |B| \) and \( 1 \notin B \).  Then there exists a matching from \( A \) to \( B \) if and only if for every pair of nonempty subsets \( S \subseteq A \) and \( R \subseteq B \cup \{ 1 \} \) such that \( SR = S \), we have \( |S| \leq |B \setminus R| \).
\end{theorem}

\begin{proof}
Assume that there is a matching from \( A \) to \( B \).  Suppose \( S \) and \( R \) are nonempty sets satisfying \( S \subseteq A \), \( R \subseteq B \cup \{ 1 \} \), and \( SR = S \). We will show that \( |S| \leq |B \setminus R|\).

Let \( U = \{ b \in B : Sb \subseteq A \} \). Since there is a matching from \( A \) to \( B \), it follows from Corollary~\ref{RevisedPhilipHall} that
\[ |S| \leq |B \setminus U|. \]  
Observe that \( R \) is a subset of \( U \cup \{ 1 \} \), since \( R \subseteq B \cup \{ 1 \} \) and \( SR = S \subseteq A \). Hence 
\[ B \setminus (U \cup \{ 1 \}) \subseteq B \setminus R. \]
We have \( 1 \notin B \), so this inclusion can be simplified to \( B \setminus U \subseteq B \setminus R, \) which implies
\[ |B \setminus U| \leq |B \setminus R|. \]
Combining our inequalities gives \( |S| \leq |B \setminus R| \), as desired.

Assume, conversely, that the condition in the statement involving \( S \) and \( R \) holds. We will show that there is a matching from \( A \) to \( B \) by verifying that the condition in Corollary~\ref{RevisedPhilipHall} is satisfied.

Let \( S \) be a nonempty subset of \( A \). As above, define \( U = \{ b \in B : Sb \subseteq A \}. \) 
We will show that \( |S| \leq |B \setminus U| \). Let \( R = U \cup \{1\} \). We consider two cases.

\medskip

\textbf{Case 1:} \( SR = S \). We can apply our hypothesis to \( S \) and \( R \), to obtain
\[
|S| \leq |B \setminus R|.
\]
Since \( 1 \notin B \), we have \( B \setminus R = B \setminus U \), and so
\[
|S| \leq |B \setminus U|,
\]
as desired.

\medskip

{\bf Case 2:} \( SR \neq S\).  We will employ Dyson's \( e \)-transform.  Let \( e \in S \) and \( r \in R \) be such that \( er \notin S \).  Define sets \( S_1 \) and \( R_1 \) by
\begin{align*}
 S_1 &= S \cup (eR), \\
R_1 &= R \cap (Se^{-1}).
\end{align*}
We claim that the following conditions hold:
\begin{itemize}
\item[(i)] \( S_1R_1 \subseteq SR \subseteq A \),
\item[(ii)] \( |S_1| + |R_1| = |S| + |R| \),
\item[(iii)] \( 1 \in R_1 \subseteq R \subseteq B \cup \{ 1 \}\),
\item[(iv)] \( S_1 \subseteq A \mbox{ and }|S| < |S_1|. \)
\end{itemize}
Conditions (i) and (iii) follow directly from the definitions of \( S \), \( R \), \( S_1 \), and \( R_1 \). Regarding (iv), we have \( S_1 \subseteq A \) on account of (i) and (iii) (specifically, \( S_1R_1 \subseteq A \) and \( 1 \in R_1 \)).  The rest of (iv) follows from the fact that \( S \subseteq S_1 \) and \( er \in S_1 \setminus S \).  Finally, to see that (ii) holds, observe that 
\begin{align*}
|S_1| &= |S \cup (eR)| \\
&= |S| + |eR| - |S \cap (eR)| \\
&= |S| + |R| - |S \cap (eR)|,
\end{align*}
and the map 
\[ R_1 \longrightarrow S \cap (eR) \]
\[ x \mapsto xe \]
is a bijection. 

If \( S_1R_1 \neq S_1 \), we repeat the above, replacing \( S \) with \( S_1 \) and \( R \) with \( R_1 \).  The process continues until we reach nonempty sets \( S_m \) and \( R_m \) satisfying \( S_mR_m = S_m \). This must eventually occur because the sets \( S_1, S_2, \ldots \) are strictly increasing in size and are contained in the finite set \( A \).  Note that the sets \( S_m \) and \( R_m \) must satisfy

\begin{itemize}
\item[(v)] \( S_mR_m = S_m \subseteq A \),
\item[(vi)] \( |S_m| + |R_m| = |S| + |R| \),
\item[(vii)] \( 1 \in R_m \subseteq R \subseteq B \cup \{ 1 \}\).
\end{itemize}
Applying our hypothesis to \( S_m, R_m \) (note that (v) and (vii) ensure that \( S_m , R_m \) can play the roles of \( S , R \)), we get
\[ |S_m| \leq |B \setminus R_m|.\]
Since \( 1 \notin B \), we can rewrite this inequality as 
\[ |S_m| \leq |B \setminus (R_m \setminus \{ 1 \})|.\]
We have \( 1 \in R_m \) and \( R_m \setminus \{ 1 \} \subseteq R \setminus \{ 1 \} = U \subseteq B\), hence 
\begin{align*}
|S_m| &\leq |B| - |R_m \setminus \{ 1 \}| \\
& = |B| - |R_m| + 1.
\end{align*}
Finally, using (vi) and bearing in mind that \( |R| - 1 = |U| \), we obtain
\[ |S| \leq |B| - |U| = |B \setminus U|, \] 
so that the condition in Corollary~\ref{RevisedPhilipHall} holds.
\end{proof}

The lemma below recalls a known interpretation of the condition \( SR = S \) in Theorem~\ref{CharForAbelianGroups} in terms of cosets. It will be used frequently in the applications.

\begin{lemma} \label{UnionOfCosets}
Let \( G \) be an abelian group and let \( S \) and \( R \) be nonempty finite subsets of \( G \). Then \( SR = S \) if and only if \( S \) is a union of cosets of \( \langle R \rangle \). 
\end{lemma}

\begin{proof}
Assume that \( SR = S \).  Let \( a \in S \).  For any \( x \in R \), we have \( ax \in SR = S \) and by induction  \( ax^k \in S \) for all positive integers \( k \). Since \( S \) is finite, it follows that \( o(x) < \infty \).  Thus \( \langle x \rangle = \{ x, x^2, \ldots ,x^{o(x)} \} \) and clearly \( a \langle x \rangle \subseteq S \).  Note, in particular, that \( a x^{-1} = a x^{o(x)-1} \in S \).

Now suppose \( x_1^{\epsilon_1}x_2^{\epsilon_2} \cdots x_n^{\epsilon_n} \) is a word on \( R \), with each \( \epsilon_i \) equaling \( \pm 1 \), and \( a' \in S \).  By the above paragraph, \( a'x_1^{\epsilon_1} \in S \), hence \( (a'x_1^{\epsilon_1})x_2^{\epsilon_2} \in S \), and so on, giving us \( a'x_1^{\epsilon_1}x_2^{\epsilon_2} \cdots x_n^{\epsilon_n} \in S \). Thus \( a' \langle R \rangle \subseteq S \).  From this we see that \( S \) is a union of cosets of \( \langle R \rangle \).

The converse is clear.
\end{proof}

Let \( G \) be an abelian group. A subset of \( G \) of the form \( \{ a, ax, \ldots , ax^{n-1} \} \), where \( a, x \in G \) and \( n \) is a positive integer such that \( n -1 < o(x) \), is called a {\em progression} of {\em length} \( n \). 

The following result is new and can be viewed as a generalization of Theorem 1.2-(6) in \cite{Aliabadi 2}.

\begin{corollary} \label{ProgAndOrder}
Let \( A \) and \( B \) be nonempty finite subsets of an abelian group \( G \) such that \( |A| = |B| \) and \( 1 \notin B \). Let \( n \) be a positive integer. Assume that \( A \) contains no progression of length greater than \( n \), and every element of \( B \) has order greater than \( n \).  Then there exists a matching from \( A \) to \( B \). 
\end{corollary}

\begin{proof}
Suppose \( S \) and \( R \) are nonempty sets such that \( S \subseteq A \), \( R \subseteq B \cup \{ 1 \} \), and \( SR = S \).  We claim that \( R = \{ 1 \} \).  Let \( x \in R \) and \( a \in S \). Then \( a \langle R \rangle \subseteq  S \) by Lemma~\ref{UnionOfCosets}. Also, we must have \( o(x) \leq n \); otherwise, the set \( \{ ax, ax^2, \ldots , ax^{n + 1} \} \) would be a progression of length greater than \( n \) contained in \( a \langle R \rangle \subseteq  S \subseteq A \). On the other hand, every element of \( B \) is assumed to have order greater than \( n \), so \( x \) cannot belong to  \( B \). Since \( R \subseteq B \cup \{ 1 \} \), we then must have \( x = 1 \). The claim is established.

It follows that \( |B \setminus R | = |B \setminus \{ 1 \}| = |B| = |A| \geq |S| \).  We now apply Theorem~\ref{CharForAbelianGroups} to complete the proof. \end{proof}

A {\em Chowla subset} of a (not necessarily abelian) group \( G \) is a nonempty subset \( S \) with the property that every element of \( S \) has order greater than \( |S| \). In \cite{Hamidoune}, Hamidoune used the isoperimetric method to prove that if \( A \) and \( B \) satisfy the usual conditions and, in addition, \( B \) is a Chowla subset of \( G \), then there is a matching from \( A \) to \( B \).  In the case where \( G \) is abelian, this result follows easily from our work above.

\begin{corollary}\label{Chowla set}
Let \( A \) and \( B \) be nonempty finite subsets of an abelian group \( G \) such that \( |A| = |B| \) and \( 1 \notin B \). If \( B \) is a Chowla subset of \( G \), then there exists a matching from \( A \) to \( B \).
\end{corollary}

\begin{proof}
Take \( n = |A| = |B| \) in Corollary~\ref{ProgAndOrder}.
\end{proof}

\begin{remark}
    We mention that Corollary~\ref{Chowla set} can be used to derive Corollary~3.6 in \cite{Aliabadi 1}. Let \( A \) and \( B \) be nonempty finite subsets of an abelian group \( G \) such that \( 1 \notin B \) and \( |A| = |B| = n < n(G) \), where \( n(G) \) denotes the smallest cardinality of a nontrivial subgroup of \( G \). In this situation, \( B \) is a Chowla subset, and thus, by Corollary~\ref{Chowla set}, there is a matching from \( A \) to \( B \).
\end{remark}

Next, we provide a short proof, using Theorem~\ref{CharForAbelianGroups}, of a result which first appeared in \cite{Losonczy 2}. We point out that, in the argument below, the verification of the condition in Theorem~\ref{CharForAbelianGroups} involving \( S \) and \( R \) is rather different from the one given for Corollary~\ref{ProgAndOrder}.

\begin{corollary}\label{sym abelian}
Let \( A \) be a nonempty finite subset of an abelian group \( G \) such that \( 1 \notin A \). Then there exists a matching from \( A \) to itself.
\end{corollary}

\begin{proof}
Suppose \( S \) and \( R \) are nonempty sets such that \( S \subseteq A \), \( R \subseteq A \cup \{ 1 \} \), and \( SR = S \).  Then \( S \) and \( R \) are disjoint.  To see this, assume the contrary and let \( a \in S \cap R \). Then \( \langle R \rangle = a \langle R \rangle \) and, by Lemma~\ref{UnionOfCosets}, \( a \langle R \rangle \subseteq S \).  Hence \( 1 \in S \subseteq A \), a contradiction.

By the above and the fact that \( S \cup (R \setminus \{ 1 \} ) \subseteq A \), we have \( |S| + |R \setminus \{ 1 \} | \leq |A| \), and so 
\[ |S| \leq |A| - |R \setminus \{ 1 \} | = |A \setminus (R \setminus \{ 1 \} )| = |A \setminus R|. \]  
Applying Theorem~\ref{CharForAbelianGroups} completes the proof.
\end{proof}

Let \( A \) be a finite subset of an abelian group \( G \).  We say that \( A \) is a {\em Sidon set} if every \( x \) in \( G \) can be written in at most one way as a product \( x = a_1a_2 \), with \( a_1, a_2 \in A \), up to a transposition of the factors.  It was shown in \cite{Aliabadi 2} (see Theorem 1.2-(5)) that if \( A \subseteq G \) is a nonempty Sidon set, then for any subset \( B \) of \( G \) of the same size as \( A \) with \( 1 \notin B \), there is a matching from \( A \) to \( B \).  Below, we establish a lower bound for the number of such matchings.

\begin{corollary}
Let \( A \) and \( B \) be nonempty finite subsets of an abelian group \( G \) such that \( |A| = |B| \) and \( 1 \notin B \). Assume that \( A \) is a Sidon set. Then there are at least \( (|A| - 1)! \) matchings from \( A \) to \( B \).
\end{corollary}

\begin{proof}
To estimate the number of matchings, we will apply Theorem~\ref{MarshallHall}. By Theorem 1.2-(5) in \cite{Aliabadi 2}, we know that there is at least one matching, but we will prove this fact here in a different way in order to give another example of the applicability of Theorem~\ref{CharForAbelianGroups}. Suppose \( S \) and \( R \) are nonempty sets such that \( S \subseteq A \), \( R \subseteq B \cup \{ 1 \} \), and \( SR = S \). We will show that \( R = \{ 1 \} \).  Let \( x \in R \) and \( a \in S \). By Lemma~\ref{UnionOfCosets}, \( a \langle R \rangle \) is a subset of \( S \), and hence of \( A \).

Observe that \( o(x) \leq  2 \), since otherwise \( a\), \( ax \), \( ax^2 \) would be distinct elements of \( a \langle R \rangle \) satisfying \( (ax)(ax) = (ax^2)a \), contradicting the Sidon property of \( A \).  In fact, we cannot have \( o(x) = 2 \) because this would mean that the elements \( a \neq ax \) satisfy \( aa = (ax)(ax) \), another contradiction. Thus \( R = \{ 1 \} \). By Theorem~\ref{CharForAbelianGroups}, there is a matching from \( A \) to \( B \).

To use Theorem~\ref{MarshallHall}, we also need to show that \( d_2(b) \geq |A| - 1 \) for each \( b \in B \). Assume the contrary. Then there exist \( b \in B \) and distinct \( a_1 , a_2 \in A \) such that \(  a_1b, a_2b \in A \).  Let \( y_1 = a_1b \) and \( y_2 = a_2b \).  Then \( b = a_1^{-1} y_1 = a_2^{-1} y_2 \), hence \( a_2 y_1 = a_1 y_2 \) (note that  \( a_1 \neq y_1 \), since \( b \) cannot equal \( 1 \)).  This contradiction to the Sidon assumption completes the proof.
\end{proof}

An abelian group \( G \) is said to have the {\em matching property} if for all pairs of nonempty finite subsets \( A \) and \( B \) of \( G \) with \( |A| = |B| \) and \( 1 \notin B \), there exists a matching from \( A \) to \( B \).  The result below, which first appeared in \cite{Losonczy 2}, characterizes the groups \( G \) having the matching property. The original proof relied on a theorem of Kneser; the one that follows uses Theorem~\ref{CharForAbelianGroups}.

\begin{corollary}  \label{MatchingPropertyGroup}
Let \( G \) be an abelian group.  Then \( G \) has the matching property if and only if \( G \) is torsion-free or of prime order.  
\end{corollary}

\begin{proof}
First observe that the trivial group has the matching property and is torsion-free, so we may assume \( |G| > 1 \) in what follows.

Suppose \( G \) is torsion-free or of prime order, and let \( A \) and \( B \) be nonempty finite subsets of \( G \) such that \( |A| = |B| \) and \( 1 \notin B \).  We will use Theorem~\ref{CharForAbelianGroups} to show that there is a matching from \( A \) to \( B \).

As usual, let \( S \) and \( R \) be nonempty sets such that  \( S \subseteq A \), \( R \subseteq B \cup \{ 1 \} \), and \( SR = S \). Let \( a \in S \) and note that \( a \langle R \rangle \subseteq S \) by Lemma~\ref{UnionOfCosets}. Hence 
\[ | \langle R \rangle | = |a \langle R \rangle | \leq |S| \leq |A|, \] 
from which we see that \( \langle R \rangle \) is finite and unequal to \( G \).  By hypothesis, \( G \) has no nontrivial proper finite subgroups, whence \( \langle R \rangle = \{ 1 \} \). We can now apply Theorem~\ref{CharForAbelianGroups}.

Conversely, assume that \( G \) is neither torsion-free nor of prime order.  Then \( G \) has a nontrivial proper finite subgroup \( H \). Choose \( g \in G \setminus H \) and define \( A = H \) and \( B = (H \setminus \{ 1 \} ) \cup \{ g \} \).  Then \( 2 \leq  |A| = |B| < \infty \), \( 1 \notin B \), and there is no matching from \( A \) to \( B \), since every \( b \in B \setminus \{ g \} \) satisfies \( Hb = H \).
\end{proof}

For our final application of Theorem~\ref{CharForAbelianGroups}, we present a generalization of a result on the existence of matchings which appeared recently in  \cite{Aliabadi 2} (see Theorem 1.2-(7)).

\begin{corollary}\label{large sumsets}
Let \( A \) and \( B \) be nonempty finite subsets of an abelian group \( G \) such that \( |A| = |B| \) and \( 1 \notin B \). Assume that for every \( a \in G \) and every nontrivial proper finite subgroup \( H \) of \( G \), we have 
\[ |aH \cap A| + |H \cap B| < |H| + 1. \]
Then there exists a matching from \( A \) to \( B \).
\end{corollary}

\begin{proof}
Suppose \( S \) and \( R \) are nonempty sets satisfying \( S \subseteq A \), \( R \subseteq B \cup \{ 1 \} \), and \( SR = S \). Let \( a \in S \) and observe, as before, that \( a \langle R \rangle \subseteq S \) by Lemma~\ref{UnionOfCosets}. Hence \( \langle R \rangle \) is finite and proper.  Assume for a contradiction that \( \langle R \rangle \) is nontrivial, and note that this implies \( \langle R \rangle \cap B \neq \emptyset \). Applying our hypothesis (taking \( H = \langle R \rangle \)), we get 
\[ |a \langle R \rangle \cap A| + | \langle R \rangle \cap B| < | \langle R \rangle | + 1. \]
Since \( a \langle R \rangle \subseteq A \), this simplifies to
\[ |\langle R \rangle| + | \langle R \rangle \cap B| < | \langle R \rangle | + 1, \]
forcing \( \langle R \rangle \cap B = \emptyset \), a contradiction. We thus have \( R  = \{ 1 \} \), enabling us to apply Theorem~\ref{CharForAbelianGroups}.
\end{proof}

\section{Matching subspaces in a field extension}\label{MSFE}

We first adopt the following conventional notation. For a $K$-vector space $V$, the dual space of \( V \) is denoted by \( V^* \).  Thus, 
\[
V^* = \{ f : V \to K \mid f \text{ is a \( K \)-linear mapping} \}.
\]
For any subspace \( W \subseteq V \), we define its annihilator \( W^\perp \) in \( V^* \) by
\[
W^\perp = \{ f \in V^* \mid W \subseteq \ker f \}.
\]
It is a standard result that if \( V \) is finite dimensional, 
\[
\dim W^\perp = \dim V - \dim W.
\]

The following is our second main result. It is the linear counterpart to Theorem~\ref{CharForAbelianGroups}. We note that the approach taken here parallels that in the abelian group setting, where Dyson’s $e$-transform plays a central role. In the proof below, we introduce and employ a linear analogue of the $e$-transform, which, to the best of our knowledge, has not been previously investigated and may be of independent interest.

\begin{theorem}\label{main linear}
    Let \( K \subsetneq L \) be a field extension, and let \( A \) and \( B \) be two \( n \)-dimensional \( K \)-subspaces of \( L \), with \( n > 0 \) and \( 1 \notin B \). Then \( A \) is matched to \( B \) if and only if for every pair of nonzero \( K \)-subspaces \( S \subseteq A \) and \( R \subseteq B\oplus K \) with \( \langle SR \rangle = S \), we have
\[
\dim S\leq \dim (B/(B\cap R)).
\]
\end{theorem}

\begin{proof}
   Assume that $A$ is matched to $B$. Suppose $S$ and $R$ are nonzero subspaces satisfying $S \subseteq A$ and $R \subseteq B \oplus K$ with \( \langle SR \rangle = S \). Let $\mathcal{S} = \{a_1, \ldots, a_\ell\}$ be a basis for $S$. We will show that $\dim (R\cap B) \leq n - \ell$. Extend $\mathcal{S}$ to a basis $\mathcal{A} = \{a_1, \ldots, a_\ell, a_{\ell+1}, \ldots, a_n\}$ for \( A \). Since $A$ is matched to $B$, there exists a basis $\mathcal{B} = \{b_1, \ldots, b_n\}$ for $B$ such that $\mathcal{A}$ is matched to $\mathcal{B}$. Thus,
   \[
   a_i^{-1}A \cap B \subseteq \langle b_1, \ldots, b_{i-1}, b_{i+1}, \ldots, b_n \rangle  \quad \text{for each }i \in [n].
   \]
   This implies
   \begin{align}\label{eq1}
       \bigcap_{i \in [\ell]} (a_i^{-1}A \cap B) \subseteq \bigcap_{i \in [\ell]} \langle \mathcal{B} \setminus \{b_i\} \rangle = \langle b_{\ell+1}, \cdots, b_n \rangle.
   \end{align}
   On the other hand, since \( \langle SR \rangle = S \), one has \( a_i R \subseteq S \) for each $i\in[\ell]$, and hence \( R \subseteq a_i^{-1} S \subseteq a_i^{-1} A \). Therefore,
\[
R\cap B \subseteq \bigcap_{i \in [\ell]} (a_i^{-1} A \cap B).
\] 
Combining this with (\ref{eq1}), we have $\dim (R \cap B) \leq n - \ell$, which implies  
\[
\dim S\leq \dim (B/(R\cap B)),
\]
as desired.

\bigskip

   Conversely, assume that the condition in the statement involving \( S \) and \( R \) holds. We will show that $A$ is matched to $B$. Let $\mathcal{A}=\{a_1,\ldots, a_n\}$ be a basis for $A$. Let $J\subseteq [n]$ be nonempty. Define $S=\langle a_i:i\in J \rangle$, $T= \bigcap_{i \in J} (a_i^{-1}A \cap B)$ and $R=T\oplus K$. We claim that $\dim T \leq n - |J|$. Our argument splits into two cases.
   
\bigskip

\textbf{Case 1:} \( \langle SR \rangle = S \). By our hypothesis, we have
\begin{align*}
  \dim S\leq \dim ( B \setminus (R \cap B)).  
\end{align*}
Since \( T \subseteq B \), \( R = T \oplus K \) and \( 1 \notin B \), it follows that \( T = R \cap B \), and so  
\[
\dim T=\dim (R \cap B)\leq n-|J|,
\]
as claimed.

\bigskip

\textbf{Case 2:} \( \langle SR \rangle \neq S \).  Then \( S\subsetneq SR \). Choose \( e \in S \) and \( r \in R \) such that \( er \in SR\setminus S \). Define subspaces \( S_1 \) and \( R_1 \) as follows:
\begin{align*}
S_1 &= S + eR, \\
R_1 &= R \cap (Se^{-1}).
\end{align*}
We claim that the following conditions hold:

\begin{enumerate}
    \item[(i)] \( \langle S_1 R_1\rangle \subseteq \langle SR\rangle \subseteq A \).
    \item[(ii)] \( \dim S_1 + \dim R_1 = \dim S + \dim R \).
    \item[(iii)] \( 1 \in R_1 \subseteq R \subseteq B \oplus K \).
    \item[(iv)] \( S_1 \subseteq A \) and \( \dim S < \dim S_1 \).
\end{enumerate}
Condition (iii) and the first inclusion in (i) follow directly from the definitions of \( S \), \( R \), \( S_1 \), and \( R_1 \). 

To verify (ii), we first apply the dimension of a sum formula for vector subspaces:
\begin{align*}
\dim S_1 &= \dim(S + eR) \\
         &= \dim S + \dim eR - \dim(S \cap eR) \\
         &= \dim S + \dim R - \dim(S \cap eR).
\end{align*}
Now observe that the map \( x \mapsto xe \) defines a linear isomorphism from \( R \cap (S e^{-1}) \) to \( S \cap (eR) \), since \( e \neq 0 \). This implies
\[
\dim R_1 = \dim(R \cap (S e^{-1})) = \dim(S \cap eR),
\]
so
\[
\dim S_1 + \dim R_1 = \dim S + \dim R,
\]
confirming (ii).

Concerning (iv) and the second inclusion in (i), we proceed as follows. By construction, we have
\[
S T \subseteq A.
\]
Since \( R = T \oplus K \), it follows that
\[
S R = S (T \oplus K) \subseteq \langle S T \cup S \rangle \subseteq A,
\]
and hence \( \langle SR \rangle \subseteq A \).  In particular,
\[
eR \subseteq SR \subseteq A.
\]
Also, we have \( S \subseteq A \) by the definition of \( S \). Therefore,
\[
S_1 = S + eR \subseteq A.
\]
Moreover, \( S_1 \) properly contains \( S \), since \( e r \in S_1 \setminus S \). Hence
\[
\dim S < \dim S_1.
\]
All four conditions have been verified. Now if \( \langle S_1 R_1 \rangle \neq S_1 \), we repeat the above procedure, substituting \( S_1 \) for \( S \) and \( R_1 \) for \( R \). This iterative process continues until we obtain nonzero subspaces \( S_m \) and \( R_m \) such that \( \langle S_m R_m \rangle = S_m \). Termination is guaranteed, as \( A \) is finite-dimensional and the sequence of subspaces \( S_1, S_2, \ldots \) strictly increases in dimension.

At the final step, the subspaces \( S_m \) and \( R_m \) satisfy:
\begin{itemize}
\item[(v)] \( \langle S_m R_m \rangle = S_m \subseteq A \),
\item[(vi)] \( \dim S_m + \dim R_m = \dim S + \dim R \),
\item[(vii)] \( 1 \in R_m \subseteq R \subseteq B \oplus K \).
\end{itemize}

By applying our hypothesis to \( S_m \) and \( R_m \), which is justified by (v) and (vii), we obtain
\[
\dim S_m \leq \dim (B\setminus (R_m \cap B)).
\]
Since \( \dim (R_m \cap B) = \dim R_m - 1 \), it follows that 
\[
\dim S_m + \dim R_m \leq n + 1,
\]
which, combined with (vi) and the fact that \( \dim R = \dim T + 1 \), yields
\[
\dim T \leq n - |J|,
\]
as claimed.
\bigskip

So in both cases we have \( \dim T \leq n - |J| \). Passing to the annihilator in the dual space \( B^* \), we obtain
\[
\dim T^\perp \geq |J|,
\]
which leads to
\[
\dim \left( \sum_{i \in J} (a_i^{-1} A \cap B)^\perp \right) \geq |J|.
\]
Applying Theorem~\ref{Linear Hall} to the family \( \{(a_i^{-1} A \cap B)^\perp\}_{i \in [n]} \), we obtain a free transversal \( \{f_1, \ldots, f_n \}  \subseteq B^* \) such that 
\begin{align}\label{eq2}
f_i \in (a_i^{-1} A \cap B)^\perp  \quad \text{for each }i \in [n].
\end{align}
Note that \( \{f_1, \ldots, f_n\} \) is a basis for \( B^* \).

Let \( \mathcal{B} = \{b_1, \ldots, b_n\} \) be the basis of \( B \) dual to \( \{f_1, \ldots, f_n\} \).  We show that \( \mathcal{A} \) is matched to \( \mathcal{B} \). Observe, \( f_i(b_j) = \delta_{ij} \), and so
\[
\ker f_i = \langle b_1, \ldots, b_{i-1}, b_{i+1}, \ldots, b_n \rangle \quad \text{for each }i \in [n].  
\]
This combined with (\ref{eq2}) gives us 
\[
a_i^{-1} A \cap B \subseteq \langle b_1, \ldots, b_{i-1}, b_{i+1}, \ldots, b_n \rangle \quad \text{for each } i \in [n],
\]
as desired. Therefore, \( A \) is matched to \( B \), completing the proof.
\end{proof}

We will need some basic notation from field theory. If \( K \subseteq L \) is a field extension and \( x \in L \), we write \( K(x) \) for the subfield of \( L \) generated by \( K \cup \{ x \} \), and \( [K(x) : K ] \) for the dimension of \( K(x) \) as a vector space over \( K \). If \( x \) is algebraic over \( K \), we write \( m_x(t) \) for its minimal polynomial; recall that \( \deg m_x(t) = [K(x) : K ] \).

The following lemma provides insight into how the condition \( \langle SR \rangle = S \) from Theorem~\ref{main linear} will be used in applications.

\begin{lemma} \label{Linear Lemma}
Let \( K \subseteq L \) be a field extension, let \( n \) be a positive integer, let \( S \) and \( R \) be nonzero \( K \)-subspaces of \( L \), and let \( x \in R \).  Assume that \( \langle SR \rangle = S \) and  \( \dim S \leq n \). Then \( x \) is algebraic over \( K \) and in fact \( [K(x): K] \leq n \). Also, for each \( a \in S \), we have \( aK(x) \subseteq S \).  
\end{lemma}

\begin{proof}
Suppose \( a \in S \). Note that both conclusions hold when \( x = 0 \), and the second conclusion holds when \( a = 0 \). Assume \( x, a \neq 0 \).  

We have \( ax \in \langle SR \rangle = S \) and by induction \( ax^k \in S \) for all positive integers \( k \). Since \( \dim S \leq n \), there exist scalars \( c_1, \ldots , c_{n+1} \in K \), not all \( 0 \), such that 
\[
\sum_{i=1}^{n+1}c_iax^i = 0.
\]
Multiplying through by \( (ax)^{-1} \) gives us
\[
\sum_{i=1}^{n+1}c_ix^{i-1} = 0,
\]
which shows that \( x \) is algebraic over \( K \) and moreover that the minimal polynomial \( m_x(t) \) has degree at most \( n \). Thus \( [K(x):K] \leq n \). 

For the last part of the lemma, note that \( x^k \in a^{-1}S \) for all \( k \geq 1 \), and also \(1 \in a^{-1} S \) because \( a \in S \).  Since every element of \( K(x) \) can be written in the form \( p(x) \) for some polynomial \( p(t) \) in \( K[t] \), it follows that \( K(x) \) is contained in the \( K \)-subspace \( a^{-1}S \), and hence \( aK(x) \subseteq S \). 
\end{proof}

Let \( K \subseteq L \) be a field extension, and let \( A \) be a \( K \)-subspace of \( L \). We say that \( A \) is a \emph{Chowla subspace} if for every \( a \in A \setminus \{0\} \), we have
\[
[K(a) : K] \geq \dim A + 1.
\]

Note that if \( A \) is a Chowla subspace, then \( 1 \notin A \).

The above definition first appeared in~\cite{Aliabadi 2} and was motivated by Hamidoune’s findings~\cite{Hamidoune} concerning matchable subsets \( A \) and \( B \) of a group, where \( B \) is a Chowla subset. Also in~\cite{Aliabadi 2}, a potential for matchings in the linear setting, with \( B \) a Chowla subspace, was conjectured (Conjecture~5.2). The result below provides an affirmative answer to this conjecture.

\begin{corollary}\label{linear Chowla}
    Let \( K \subsetneq L \) be a field extension, and let \( A \) and \( B \) be two \( n \)-dimensional \( K \)-subspaces of \( L \), with \( n > 0 \). Assume that \( B \) is a Chowla subspace. Then \( A \) is matched to \( B \).
\end{corollary}

\begin{proof}
We will use Theorem~\ref{main linear}. Suppose \( S \subseteq A \) and \( R \subseteq B \oplus K \) are nonzero subspaces such that \( \langle SR \rangle = S \). 

\medskip

We claim that \( R = K \).  To see this, assume the contrary. Choose a nonzero element \( a \in S \) and an element \( x \in R \setminus K \). Note that, since \( S \) is contained in \( A \), we have \( \dim S \leq n \). By Lemma~\ref{Linear Lemma}, 
\[ [K(x):K] = \deg m_x(t) \leq n. 
\] 
On the other hand, since \( x \) lies in \( R \subseteq B \oplus K \) but not in \( K \), we can write \( x = b + c \), where \( b \in B \setminus \{ 0 \} \) and \( c \in K \). Note that \( b \) must be algebraic over \( K \), since \( x \) and \( c \) are, and moreover \( m_b(t) \) has the same degree as \( m_x(t) \) since \( m_b(t) = m_x(t + c) \).  We now use the fact that \( B \) is a Chowla subspace to obtain
\[
[K(x) : K ] = [K(b) : K ] \geq n + 1,
\]
a contradiction.  The claim is proved.

\medskip

Since \( 1 \notin B \), the claim gives us \( R \cap B = \{ 0 \} \).  Therefore, \( \dim S \leq n = \dim (B / (R \cap B)) \). We apply Theorem~\ref{main linear} to complete the proof.
\end{proof}

\begin{remark}
   Note that Corollary~\ref{linear Chowla} can be used to derive the commutative case of Theorem~5.5 from \cite{Eliahou 2}, which addresses the matchability of small subspaces.\footnote {We note that the results in \cite{Eliahou 2} cited in this paper are stated and proved in the more general setting of a skew field extension \( K \subseteq L \), where \( K \) is contained in the center of \( L \).}

Let \( K \subsetneq L \) be a field extension, and let \( A, B \subseteq L \) be \( n \)-dimensional \( K \)-subspaces with \( 1 \notin B \). Suppose \( n < n_0(K, L) \), where \( n_0(K, L) \) denotes the smallest degree of an intermediate field extension \( K \subsetneq F \subseteq L \). Under this assumption, \( B \) is a Chowla subspace of \( L \), and by Corollary~\ref{linear Chowla}, the subspace \( A \) is matched to \( B \).
\end{remark}

Next, we use Theorem~\ref{main linear} to recover, in the commutative setting, Theorem 2.8 in \cite{Eliahou 2}. This result is a linear analogue of Corollary~\ref{sym abelian}.

\begin{corollary}\label{sym lin match}
    Let \( K \subsetneq L \) be a field extension, and let \( A \) be a nonzero finite-dimensional \( K \)-subspace of \( L \). Then \( A \) is matched to itself if and only if \( 1 \notin A \).
\end{corollary}

\begin{proof}
Assume that \( 1 \notin A \). Suppose \( S \subseteq A \) and \( R \subseteq A \oplus K \) are nonzero \( K \)-subspaces such that \( \langle SR \rangle = S \). To apply Theorem~\ref{main linear}, we need to show that
\[
\dim S \leq \dim (A/(R \cap A)).
\]

\medskip

We claim that \( S \cap R = \{ 0 \} \).  Assume the contrary and let \( x \) be a nonzero element of \( S \cap R \).  Note that \( S \) is finite dimensional, since it is contained in \( A \). By Lemma~\ref{Linear Lemma},  \( x \) is algebraic over \( K \). We write out its minimal polynomial as follows:
\[ m_x(t) = c_0 + c_1 t + \cdots + c_{n-1} t^{n-1} + t^n,
\]
where \( n > 0 \) and \( c_0, \ldots , c_{n-1} \in K \). Plugging in \( x \) results in the equation
\[
0 = c_0 + c_1 x + \cdots + c_{n-1} x^{n-1} + x^n.
\]
Note that \( c_0 \neq 0 \) (the constant term of \( m_x(t) \) must be nonzero since \( x \neq 0 \)).  Solving the above equation for \( c_0 \) and multiplying through by \( c_0^{-1} \), we get
\[
1 = -c_0^{-1}x^n - c_0^{-1} \sum_{i=1}^{n-1} c_i x^i.
\]
The expression on the right belongs to the \( K \)-subspace \(  S \), since all positive powers of \( x \) lie in \( S \) (this follows from the equation \( \langle SR \rangle = S \) and the fact that \( x \in S \cap R \)). Thus \( 1 \in S \subseteq A \), a contradiction.  The claim is proved.

\medskip

We now compute
\begin{align*}
\dim A &\geq \dim (S + (R \cap A)) \\ 
&= \dim S + \dim (R \cap A) - \dim (S \cap (R \cap A)) \\
&= \dim S + \dim (R \cap A) - \dim (S \cap R) \\
&= \dim S + \dim (R \cap A),
\end{align*}
where the last equality follows from the claim. This gives us
\[
\dim S \leq \dim (A/(R \cap A)).
\]
By Theorem~\ref{main linear}, \( A \) is matched to itself. 

The converse was discussed in Remark~\ref{Necessary for Matching}.
\end{proof}

Another consequence of Theorem~\ref{main linear} is a characterization of field extensions with respect to the linear matching property (see Section~\ref{linear sec} to recall the definition).  The corollary below is the commutative case of Theorem~2.6 in \cite{Eliahou 2}. It can be viewed as a linear analogue of Corollary \ref{MatchingPropertyGroup}.

\begin{corollary} \label{linear match prop}
    A field extension $K\subsetneq L$ has the linear matching property if and only if  $L$ contains no nontrivial proper finite-dimensional extension over $K$.
\end{corollary}

\begin{proof}
Assume that \( L \) contains no nontrivial proper finite-dimensional extension over \( K \). Let \( A \) and \( B \) be two \( n \)-dimensional \( K \)-subspaces of \( L \), with \( n > 0 \) and \( 1 \notin B \). We aim to show that \( A \) is matched to \( B \) using Theorem~\ref{main linear}. 

To that end, suppose \( S \subseteq A \) and \( R \subseteq B \oplus K \) are nonzero \( K \)-subspaces such that \( \langle SR \rangle = S \). 

\medskip

We claim that  \( R = K \). To see this, assume the contrary and let \( x \in R \setminus K \). By Lemma~\ref{Linear Lemma}, \( [K(x):K] \leq n \). Since \( x \notin K \) and \( L \) contains no nontrivial proper finite-dimensional extension over \( K \), it follows that \( K(x) = L \).  Hence
\[
\dim L = [K(x):K] \leq n.
\]
Since \( \dim B = n \), we must have \( B = L \), contradicting \( 1 \notin B \). The claim is established. 

\medskip

Now observe that \( R \cap B = \{ 0 \} \), by the claim and the fact that \( 1 \notin B \). Hence
\[
\dim S \leq \dim A = \dim B = \dim (B/(R \cap B)).
\]
We now apply Theorem~\ref{main linear} to conclude that \( A \) is matched to \( B \).

\medskip

Conversely, assume that \( K \subseteq L \) admits a nontrivial proper finite-dimensional extension over \( K \). Then there exists an element \( a \in L \) of finite degree \( n \geq 2 \) over \( K \) such that \( K(a) \subsetneq L \). Choose an element \( x \in L \setminus K(a) \), and define the \( K \)-subspaces \( A \) and \( B \) of \( L \) by
\[
A = \langle 1, a, a^2, \dots, a^{n-1} \rangle,
\]
\[
B = \langle x, a, a^2, \dots, a^{n-1} \rangle.
\]
We will use Theorem~\ref{main linear} to show that \( A \) is not matched to \( B \). Take \( S = A \) and \( R = \langle a, a^2, \dots, a^{n-1} \rangle \). Then \( \langle SR \rangle = S \) (note that \( A = K(a) \)).

However, we have
\[
\dim S = n > 1 = \dim (B/(R \cap B)),
\]
violating the condition of Theorem~\ref{main linear}. Therefore, \( A \) is not matched to \( B \), implying that the field extension \( K \subseteq L \) does not have the linear matching property.
\end{proof}

Our final objective is to use Theorem~\ref{main linear} to establish the following linear counterpart to Corollary \ref{large sumsets}.

\begin{corollary}
Let \( K \subsetneq L \) be a field extension, and let \( A, B \subseteq L \) be \( n \)-dimensional \( K \)-subspaces of \( L \), with \( n > 0 \) and \( 1 \notin B \). Assume that for every \( a \in L \) and every nontrivial proper finite-dimensional intermediate subfield \( K \subseteq H \subseteq L \), the following inequality holds:
\begin{align*}
    \dim(aH \cap A) + \dim(H \cap B) < [H : K] + 1.
\end{align*}
Then \( A \) is matched to \( B \).
\end{corollary}

\begin{proof}
As usual, we assume \( S \subseteq A \) and \( R \subseteq B \oplus K \) are nonzero subspaces such that \( \langle SR \rangle = S \). To apply Theorem~\ref{main linear}, we need to show that
\begin{align*}
\dim S \leq \dim (B/(R \cap B)).
\end{align*}

We claim that  \( R = K \). Assume the contrary, and let \( x \in R \setminus K \). Let \( a \in S \setminus \{ 0 \} \).  Applying Lemma~\ref{Linear Lemma}, we find that \( [K(x):K]\leq n \) and \( aK(x) \subseteq S \).  Clearly \( [K(x):K] > 1 \), as well. Note that \( K(x) \neq L \), since otherwise, as in the previous proof, we would have \( B = L \), contradicting \( 1 \notin B \). We now apply our hypothesis (taking \( H = K(x) \)), to obtain 
\[
\dim(aK(x) \cap A) + \dim(K(x) \cap B) < [K(x) : K] + 1.
\]
Since \( aK(x) \subseteq S \subseteq A \), this simplifies to 
\[
\dim K(x) + \dim(K(x) \cap B) < [K(x) : K] + 1,
\]
forcing \( \dim(K(x) \cap B) < 1\).  But this is impossible, since \( x \in R \setminus K \) and \( R \subseteq B \oplus K \). The claim is proved.

\medskip

By the claim and the fact that \( 1 \notin B \), we have \( R \cap B = \{ 0 \} \), so the inequality \( \dim S \leq \dim (B / (R \cap B)) \) holds trivially.  Applying Theorem~\ref{main linear} completes the proof.
\end{proof}

\begin{remark}
It seems plausible that Theorems \ref{CharForAbelianGroups} and \ref{main linear} admit extensions to the non-commutative setting.  Possible approaches include the use of Kemperman's \( d \)-transform (see~\cite{Olson}) along with its linearization as presented in~\cite{Eliahou 3}.  Hamidoune’s isoperimetric method \cite{Hamidoune2} may also provide a viable framework for pursuing a generalization in the group setting. 

\end{remark}
\bigskip

\noindent \textbf{Data sharing:} Data sharing not applicable to this article as no datasets were generated or analysed.\\
\textbf{Conflict of interest:} To our best knowledge, no conflict of interests, whether of financial or personal nature, has influenced the work presented in this article.

\end{document}